\documentclass[11pt,a4paper]{article}
\usepackage[american]{babel}      
\usepackage[utf8]{inputenc}
\usepackage{amsmath,amssymb,amsthm}
\usepackage{graphicx}

\title{Convergence in Orlicz spaces by means of the multivariate max-product neural network operators of the Kantorovich type and applications}
         
\author{{\bf Danilo Costarelli} \hskip1.2cm {\bf Anna Rita Sambucini} \\  {\bf Gianluca Vinti}     \\    \\
   Department of Mathematics and Computer Science \\
            University of Perugia\\
        1, Via Vanvitelli, 06123 Perugia, Italy    \\  \\
  {\small {\tt danilo.costarelli@unipg.it} \hskip0.2cm {\tt anna.sambucini@unipg.it}} \\  {\small {\tt gianluca.vinti@unipg.it}} }

\date{}

\newcommand{\mau}{\geq}
\newcommand{\miu}{\leq}
\newcommand{\ep}{\varepsilon}
\newcommand{\N}{\mathbb{N}}
\newcommand{\R}{\mathbb{R}}
\newcommand{\Z}{\mathbb{Z}}
\newcommand{\disp}{\displaystyle}
\newcommand{\be}{\begin{equation}}
\newcommand{\ee}{\end{equation}}
\newcommand{\uu}{\underline{u}}
\newcommand{\xx}{\underline{x}}
\newcommand{\kk}{\underline{k}}
\newcommand{\yy}{\underline{y}}

\newcommand{\K}{K^{(M)}} 
\newcommand{\V}{\bigvee}
\newcommand{\phis}{\phi_{\sigma}}
\newcommand{\Psis}{\Psi_{\sigma}}
\newcommand{\rr}{{\cal R}}

\newtheorem{definition}{Definition}[section]

\newtheorem{theorem}[definition]{Theorem}
\newtheorem{lemma}[definition]{Lemma}

\newtheorem{corollary}[definition]{Corollary}

\begin{document}

\maketitle  

\begin{abstract}
In this paper, convergence results in a multivariate setting have been proved for a family of neural network operators of the max-product type. In particular, the coefficients expressed by Kantorovich type means allow to treat the theory in the general frame of the Orlicz spaces, which includes as particular case the $L^p$-spaces. Concrete examples of sigmoidal activation functions are discussed, for the above operators in different cases of Orlicz spaces.   Finally, concrete applications   to real world cases have been presented in both univariate and multivariate settings. In particular, the case of reconstruction and enhancement of biomedical (vascular) image has been discussed in details.
\end{abstract}

\section{Introduction}

The study of neural network type operators has been recently investigated by many authors from the theoretical point of view, see e.g., \cite{MAI1,CACH1,CACH2,CALIPA1,BBM1}.

  A particular attention has been reserved to the so-called max-product version of the above operators, which can be useful, for instance, in the applications of probability and fuzzy theory. The introduction of this approach is due to Bede, Coroianu and Gal (see e.g., \cite{COGA1,COGA2}), and recently they summarized their results in a very complete monograph \cite{BECOGA1}. In general, the max-product version of a family of linear operators are sub-linear and has better approximation properties: in many cases, the order of convergence is faster with respect to their linear counterparts.
   
   In the present paper, we extend the results proved in \cite{COSA1} to the multivariate frame: the latter is the most suitable when we deal with neural network type approximation. Indeed, neural networks have been introduced in order to study a simple model for the human brain, which is a very complicated structure composed by billions of neurons; this justify the multidimensional form of the neural networks. The behavior of the (artificial) neurons in the neuronal models is regulated by suitable activation functions which must represent the activation and the quite phases of the neurons (\cite{LIV1}). From the mathematical point of view, functions which better represent the latter fact are those with sigmoidal behavior (\cite{GI1,GI2}). A sigmoidal function (\cite{COVI3}) is any measurable function $\sigma: \R \to \R$, such that:
$$
\lim_{x \to -\infty} \sigma(x) = 0, \quad \quad \mbox{and}\quad\quad \lim_{x \to +\infty} \sigma(x) = 1.
$$

   Problems of interpolation, or more in general, of approximation are related with the topic of training a neural network by sample values belonging to a certain training set: this explain the interest of studying approximation results by means of NN operators in various contexts (\cite{KKS1}). 
   
   Indeed, results in this sense have been deeply studied for what concerns various aspects, such as convergence, order of approximation, saturation, and so on (\cite{COVI4,COVI5,COVI8}).

Here, we consider the above NN operators, with coefficients expressed by Kantorovich means in a multivariate form. It is well-know that this kind of approach based on Kantorovich-type operators is the most suitable in order to approximate not necessarily continuous data (\cite{CMV1,ANCOVI1,BACOVI1,COVI17}). For this reason, we study the above operators in the general setting of Orlicz spaces, which is a very general context, which includes, as particular cases, the $L^p$-spaces. Moreover, since images are typical examples of discontinuous data, Kantorovich-type operators have been recently applied to image processing, see \cite{ING1,ING2}.

From the theoretical results established in the present paper, we obtain, by a unifying approach, a general treatment of the problem of the approximation of both multivariate continuous and not necessarily continuous functions by means of the max-product neural network operators. Further, the results here proved extend to a more general setting those achieved by the authors in \cite{COVI7}.

  We also provide some concrete applications of the obtained results. In particular, we show how it is possible to reconstruct and even to enhance images by means of the above max-product type operators. The proposed examples involve the biomedical images concerning the vascular apparatus. More precisely, we reconstruct a CT image (computer tomography image) depicting an aneurysm of the abdominal aorta artery. Then, we show how it is possible to reconstruct and also to increase the resolution of a given image by the above operators (without losing information), in fact obtaining a detailed image which can be more useful from the diagnostic point of view. 

   Further, also applications in case of one-dimensional data modeling have been presented (see, e.g., \cite{GO1}).
   
   The above mentioned topics give a contribution to an application field which in the last years have been widely studied by means of the use of suitable neural network type models, see e.g., \cite{AG1,EP1}.

  Now, we give a plan of the paper. In Section \ref{sec2} the definition of the above operators and their main properties have been recalled, together with some notations and preliminary results. In Section \ref{sec3} the main convergence results of the paper have been proved. In order to show that the family of multivariate Kantorovich max-product operators are modularly convergent in Orlicz spaces, we adopt the following strategy: (i) we test the norm-convergence for the above operators in case of continuous functions; (ii) we prove a modular inequality for the above operators in $L^\varphi$, where $\varphi$ is a convex $\varphi$-function; finally (iii) we obtain the desired results exploiting the modular density of the continuous functions in $L^\varphi$. In Section \ref{sec4} we provided several examples of Orlicz spaces for which the above theory holds, and also some well-known examples of sigmoidal activation functions, while in Section \ref{sec5} we present the above mentioned real world applications.

\section{The multivariate max-product neural network operators of the Kantorovich type} \label{sec2}

From now on, the symbols $\lfloor \cdot \rfloor$ and $\lceil \cdot \rceil$ will
denote respectively the ``integer'' and the ``ceiling'' part of a given number. Moreover, we define:
$$
\V_{k \in {\cal J} } A_k\ :=\ \sup \left\{ A_{\kk} \in \R,\ \kk \in {\cal J}\right\},
$$
for any set of indexes ${\cal J} \subset \Z^s$.

  Let now $f:\mathcal{R} \to \R$ be a locally integrable and bounded function, where $\mathcal{R} \subset \R^s$, is the box-set of the form $\mathcal{R}:=[a_1,b_1] \times [a_2,b_2]\times \dots \times [a_s,b_s]$. Further, let $n \in \N^+$ such that $\lceil na_i \rceil \miu \lfloor nb_i \rfloor-1$, for every $i=1,...,s$. 
  
    The multivariate max-product neural network (NN) operators of Kantorovich type (see \cite{COVI7}) activated by the sigmoidal function $\sigma$, are defined by:
$$
\K_n(f,\, \xx)\ :=\ \frac{\disp \V_{\kk \in \mathcal{J}_n}\left[ n^s \int_{R^n_{\kk}}f\left( \uu\right)\, d\uu \right] \Psis(n\xx-\kk)}{\disp \V_{\kk \in \mathcal{J}_n} \Psis(n\xx-\kk)}, \hskip0.8cm \xx \in \mathcal{R},
$$
where:
$$
R^n_{\kk}\ :=\ \left[ \frac{k_1}{n}, \frac{k_1+1}{n}   \right] \times \dots \times \left[ \frac{k_s}{n}, \frac{k_s+1}{n}\right],
$$ 
for every $\kk \in \mathcal{J}_n$, and $\mathcal{J}_n$ denotes the set of indexes $\kk \in \Z^s$ such that $\lceil n a_i \rceil \miu k_i \miu \lfloor nb_i \rfloor-1$, $i=1,..., s$, for every $n \in \N^+$ sufficiently large. 

Here, the multivariate function $\Psis : \R^s \to \R$ is defined as follows:
$$
\Psis(\xx)\ :=\ \phis(x_1)\cdot \phis(x_2) \dots \cdot \phis(x_s), 
$$
with $\xx:=(x_1,x_2,..., x_s) \in \R^s$, and:
$$
\phi_{\sigma}(x)\, :=\, \frac{1}{2}\,[\sigma(x+1)-\sigma(x-1)], \hskip1cm x \in \R,
$$
for every non-decreasing sigmoidal function $\sigma: \R \to \R$, which satisfies the following conditions:
\begin{itemize}
\item[$(\Sigma 1)$] $\sigma(x)-1/2$ is an odd function;
\item[$(\Sigma 2)$] $\phis(x)$ is non-decreasing for $x<0$ and non-increasing for $x \mau 0$;
\item[$(\Sigma 3)$] $\sigma(x)={\cal O}(|x|^{-\alpha})$ as $x \to -\infty$, for some $\alpha>0$.
\end{itemize}

Let now $f$, $g : \mathcal{R} \to \R$ be bounded functions. The following properties hold for sufficiently large $n \in \N^+$ (see \cite{COVI2} for the proof in the one-dimensional case, and \cite{COVI7} in case of functions of several variables):

\vskip0.2cm

\noindent (a) if $f(\xx) \miu g(\xx)$, for each $\xx \in \mathcal{R}$, we have $\K_n(f,\xx) \miu \K_n(g,\xx)$, for every $\xx \in \mathcal{R}$ (the operators are monotone);

\vskip0.2cm

\noindent (b) $\K_n(f+g,\, \xx) \miu \K_n(f,\xx)+\K_n(g,\xx)$, $x \in \mathcal{R}$ (the operators are sub-additive);

\vskip0.2cm

\noindent (c) $\left| \K_n(f, \xx)-\K_n(g, \xx) \right| \miu \K_n(|f-g|,\, \xx)$, $\xx \in \mathcal{R}$;

\vskip0.2cm

\noindent (d)  $\K_n(\lambda f, \xx) = \lambda \K(f, \xx)$, $\lambda>0$, $\xx \in \mathcal{R}$ (the operators are positive homogeneous).

\vskip0.2cm

    In \cite{COVI3,COVI4} some important properties of the functions $\phi_{\sigma}(x)$ and $\Psis(\xx)$ have been proved. In the following lemma, we recall those that can be useful in order to prove the convergence results for the above operators in Orlicz spaces for functions of several variables.
\begin{lemma} \label{lemma1}
$($i$)$ $\phi_{\sigma}(x) \mau 0$ for every $x \in \R$, with $\phi_{\sigma}(2) 
> 0$, and moreover $\disp \lim_{x \to \pm \infty} \phi_{\sigma}(x) = 0$;

\vskip0.2cm
\noindent $($ii$)$ $\phis(0) \mau \phis(x)$, for every $x \in \R$, and $\phis(0) \miu 1/2$. Consequently, $\Psis(\underline{0}) \mau \Psis(\xx)$, for every $\xx \in \R^s$, and $\Psis(\underline{0}) \miu 2^{-s}$;

\vskip0.2cm

\noindent $($iii$)$ $\Psis(\xx)= \mathcal{O}(\|\xx\|_2^{-\alpha})$, as $\|\xx\|_2 \to +\infty$, where $\alpha$ is the positive constant of condition $(\Sigma 3)$, and $\|\cdot\|_2$ denotes the usual Euclidean norm of $\R^s$;

\vskip0.2cm

\noindent $($iv$)$ $\Psis \in L^1(\R^s)$;

\vskip0.2cm

\noindent $($v$)$ For any fixed $\xx \in \mathcal{R}$, there holds:
$$
\V_{\kk \in \mathcal{J}_n} \Psis(n \xx - \kk)\ \mau\ [\phis(2)]^s\ >\ 0.
$$
\end{lemma}

  Now, we can introduce the general setting in which we will work, i.e., the Orlicz spaces. We begin recalling the following definition.
\vskip0.1cm
  
 A function $\varphi: \R^+_0 \to \R^+_0$ which satisfies the following assumptions:
\begin{itemize}
	\item[$(\varphi 1)$] $\varphi \left(0\right)=0$, $\varphi \left(u\right)>0$ for every $u>0$;
	\item[$(\varphi 2)$] $\varphi$ is continuous and non decreasing on $\R^+_0$;
	\item[$(\varphi 3)$] $\disp \lim_{u\to +\infty}\varphi(u)\ =\ + \infty$,
\end{itemize}
is said to be a $\varphi$-function. Let $M({\cal R})$ denotes the set of all (Lebesgue) measurable functions $f:{\cal R} \to \R$. For any fixed $\varphi$-function $\varphi$, it is possible to define the functional $I^{\varphi} : M(\rr)\to [0,+\infty]$, as follows:
$$
I^{\varphi} \left[f\right] := \int_{\rr} \varphi(\left| f(\xx) \right|)\ d\xx,\ \hskip0.5cm f \in M(\rr).
$$
The functional $I^\varphi$ is called a {\em modular} on $M(\rr)$. Then, we can define the {\em Orlicz space} generated by $\varphi$ by the set:
$$
L^{\varphi}(\rr) := \left\{f \in M\left(\rr\right):\ I^{\varphi}[\lambda f]<+\infty,\ \mbox{for\ some}\ \lambda>0\right\}.
$$
Now, exploiting the definition of the modular functional $I^{\varphi}$, it is possible to introduce a notion of convergence in $L^{\varphi}(\rr)$, the so-called {\em modular convergence}.

   More precisely, a family of functions $(f_w)_{w>0} \subset L^{\varphi}(\rr)$ is said {\em modularly convergent} to a function $f \in L^{\varphi}(\rr)$ if:
\be \label{modular-convergence}
\lim_{w \to +\infty} I^{\varphi}\left[\lambda(f_w-f)\right]\ =\ 0,
\ee
for some $\lambda>0$. If (\ref{modular-convergence}) holds for every $\lambda>0$, we will say that the family $(f_w)_{w>0}$ is norm convergent to $f$. Obviously, the norm convergence is in general stronger than modular convergence. Both the above definitions become equivalent if and only if the $\varphi$-function $\varphi$, which generates $L^\varphi(\rr)$, satisfies:
$$
\varphi(2u)\ \miu\ M\varphi(u), \hskip0.8cm u \in \R^+_0 \quad \quad (\Delta_2-condition),
$$ 
for some $M>0$.

    Finally, if we denote by $C(\rr)$ the space of all continuous functions $f:\rr \to \R$, it turns out that $C(\rr) \subset L^{\varphi}(\rr)$, and moreover, there holds that $C(\rr)$ is dense in $L^{\varphi}(\rr)$ with respect to the topology induced by the modular convergence. 
  
  Similarly, denoting respectively by $C_+(\rr)$ and $L^{\varphi}_+(\rr)$ the subspaces of $C(\rr)$ and $L^{\varphi}(\rr)$ of the non-negative and almost everywhere non-negative functions, respectively, it turns out that also $C_+(\rr)$ is modularly dense in $L^{\varphi}_+(\rr)$.
  
  For more details concerning Orlicz spaces, see e.g., \cite{MUORL,MU1,BAMUVI,AV-09,AV2,AV-14,COVI6,BAKAVI1}.
  
  In conclusion, we recall some well-known convergence result for the operators $\K_n$, that will be useful in the next sections.
\begin{theorem}[\cite{COVI7}] \label{th1}
Let $f:\mathcal{R} \to \R^+_0$ be a given bounded function. Then,
$$
\lim_{n \to +\infty} \K_n(f, \xx)\ =\ f(\xx),
$$
at any point $\xx \in \mathcal{R}$ of continuity for $f$. Moreover, if $f \in C_+(\mathcal{R})$, we have:
$$
\lim_{n \to +\infty} \|\K_n(f, \cdot)\ -\ f(\cdot)\|_{\infty}\ =\ 0.
$$
\end{theorem} 
%


\section{Convergence in Orlicz spaces} \label{sec3}

From now on, we always consider convex $\varphi$-functions $\varphi$. Hence, we begin by proving the following auxiliary result.
\begin{theorem} \label{lux-conv}
Let $f \in C_+(\rr)$ be fixed. Then, for every $\lambda>0$:
$$
\lim_{n \to +\infty} I^{\varphi}\left[ \lambda\left(   \K_n(f, \cdot) - f(\cdot) \right)\right]\ =\ 0.
$$
\end{theorem}
\begin{proof}
Let $\ep>0$ be fixed. For every fixed $\lambda>0$, using the convexity of $\varphi$, and in view of Theorem \ref{th1}, we have:
$$
I^{\varphi}\left[ \lambda\left( \K_n(f, \cdot) - f(\cdot) \right)\right]\ =\ \int_\rr\varphi\left(\lambda \left| \K_n(f, \xx) - f(\xx) \right| \right)\, d\xx
$$
$$
\miu\ \int_\rr \varphi\left(\lambda \| \K_n(f, \cdot) - f(\cdot) \|_{\infty} \right)\, d\xx\ \miu\ \int_\rr \varphi\left(\lambda\, \ep \right)\, d\xx\ \miu\ \ep\, \varphi(\lambda) |\rr|,
$$
for $n \in \N^+$ sufficiently large, where $|\rr|$ denotes the Lebesgue measure of $\rr$. Hence the proof follows since $\ep>0$ is arbitrary.
\end{proof}

  Theorem \ref{lux-conv} states that the family $\K_n f$ is norm convergent to $f \in C_+(\rr)$.
Now, we need to prove a modular inequality for the above operators, in the setting of Orlicz spaces.
\begin{theorem} \label{th3}
For every $f$, $g \in L^{\varphi}_+(\rr)$, and $\lambda>0$, it turns out that:
$$
I^{\varphi}\left[ \lambda\left( \K_n(f, \cdot) - \K_n(g, \cdot) \right)\right]\ \miu\ \|\Psis\|_1\, I^{\varphi}\left[ \phis(2)^{-s}\lambda\,  \left(f(\cdot)-g(\cdot)\right) \right],
$$
for $n \in \N^+$ sufficiently large.
\end{theorem}
\begin{proof}
Using inequality (c) of $\K_n$, and Lemma \ref{lemma1} (v), we can write what follows:
$$
I^{\varphi}\left[ \lambda\left( \K_n(f, \cdot) - \K_n(g, \cdot) \right)\right]\ =\ \int_\rr \varphi\left( \lambda \left| \K_n(f, \xx) - \K_n(g, \xx)   \right|   \right)\, d\xx
$$
$$
\hskip-6.6cm \miu\ \int_\rr \varphi\left( \lambda\, \K_n(|f-g|, \xx) \right)\, d\xx
$$
$$
\miu\ \int_\rr \varphi\left( \lambda\, \phis(2)^{-s} \V_{\kk \in {\cal J}_n} \left[n^s \int_{R^n_{\kk}} |f(\uu)-g(\uu)|\, d\uu \right] \Psis(n\xx-\kk) \right)\, d\xx.
$$
Now, we observe that:
\be \label{scambio}
\varphi\left( \V_{k \in {\cal J}} A_k  \right)\ =\ \V_{k \in {\cal J}}\varphi\left( A_k  \right),
\ee
for any finite set of indexes ${\cal J} \subset \Z^s$, since $\varphi$ is non-decreasing. Thus, using (\ref{scambio}), the convexity of $\varphi$, Lemma \ref{lemma1} (ii), and the Jensen's inequality (see, e.g., \cite{COSP2}), we obtain:
$$
\hskip-5cm I^{\varphi}\left[ \lambda\left( \K_n(f, \cdot) - \K_n(g, \cdot) \right)\right]\ 
$$
$$
\miu\ \int_\rr \V_{\kk \in {\cal J}_n}  \varphi\left( \lambda\, \phis(2)^{-s} \left[n^s \int_{R^n_{\kk}} |f(\uu)-g(\uu)|\, d\uu \right] \Psis(n\xx-\kk) \right)\, dx
$$
$$
\miu\ \int_\rr \V_{\kk \in {\cal J}_n} \Psis(n\xx-\kk)\,  \varphi\left( \lambda\, \phis(2)^{-s} \left[n^s \int_{R^n_{\kk}} |f(\uu)-g(\uu)|\, d\uu \right] \right)\, d\xx
$$
$$
\hskip-0.2cm \miu\ \int_\rr d\xx \left\{\V_{\kk \in {\cal J}_n} \Psi(n\xx-\kk)\, n^s \int_{R^n_{\kk}}  \varphi\left( \lambda\, \phis(2)^{-s} |f(\uu)-g(\uu)|  \right)\, d\uu \right\}
$$
$$
\hskip-1.2cm \miu\ \int_\rr d\xx \left\{ \V_{\kk \in {\cal J}_n} \Psis(n\xx-\kk)\, n^s \int_\rr  \varphi\left( \lambda\, \phis(2)^{-s} |f(\uu)-g(\uu)|  \right)\, d\uu\right\}
$$
$$
\hskip-2.8cm =\ I^{\varphi}\left[ \lambda\, \phis(2)^{-s} \left(f(\cdot)-g(\cdot)\right) \right]\, \V_{\kk \in {\cal J}_n}  n^s\, \int_\rr  \Psis(n\xx-\kk)\, d\xx.
$$
Putting $\yy=n\xx$, and recalling that the $L^1$-norm is shift invariant, i.e.,
$$
\|\Psis\|_1\ =\ \int_{\R^s} \Psis(\yy)\ d\yy\ =\ \int_{\R^s} \Psis(\yy - \kk)\ d\yy,
$$
for every $\kk \in \Z^s$, we finally obtain:
$$
\hskip-5cm I^{\varphi}\left[ \lambda\left( \K_n(f, \cdot) - \K_n(g, \cdot) \right)\right]\ 
$$
$$
\miu\ I^{\varphi}\left[ \lambda\, \phis(2)^{-s} \left(f(\cdot)-g(\cdot)\right) \right]\, \left[ \V_{\kk \in {\cal J}_n}  \int_{\R^s} \Psis(\yy-\kk)\, d\yy \right]
$$
$$
\hskip-2.7cm =\ I^{\varphi}\left[ \phis(2)^{-s} \lambda\, \left(f(\cdot)-g(\cdot)\right) \right]\, \|\Psis\|_1,
$$
for every $n \in \N^+$ sufficiently large.
\end{proof}
Now, we are able to prove a modular convergence theorem for the multivariate Kantorovich max-product NN operators.
\begin{theorem} \label{th4}
Let $f \in L^{\varphi}_+(\rr)$ be fixed. Then there exists $\lambda>0$ such that:
$$
\lim_{n \to +\infty} I^{\varphi}\left[\lambda \left( \K_n(f, \cdot) - f(\cdot)  \right)\right]\ =\ 0.
$$
\end{theorem}
\begin{proof}
Let $\ep>0$ be fixed. Since $C_+(\rr)$ is modularly dense in $L^{\varphi}_+(\rr)$, there exists $\bar \lambda >0$ and $g \in C_+(\rr)$ such that:
\be \label{densitaa}
I^{\varphi}\left[\bar \lambda \left( f(\cdot) - g(\cdot)  \right)\right]\ <\ \ep.
\ee
Now, choosing $\lambda >0$ such that:
$$
\max \left\{ \phis(2)^{-s}3\, \lambda,\ 3\, \lambda \right\}\, \miu\ \bar \lambda,
$$
by the property of $I^\varphi$, using Theorem \ref{th3}, and the convexity of $\varphi$, we can write what follows:
\vskip0.1cm
$$
\hskip-7cm I^{\varphi}\left[\lambda \left( \K_n(f, \cdot) - f(\cdot)  \right)\right]\ 
$$
$$
\hskip0.8cm \miu\ {1 \over 3} \left\{I^{\varphi}\left[3 \lambda \left( \K_n(f, \cdot) - \K_n(g, \cdot)  \right)\right] + I^{\varphi}\left[3 \lambda \left( \K_n(g, \cdot) - g(\cdot)  \right)\right] \right.
$$
$$
\hskip-7.2cm +\left. \ I^{\varphi}\left[3 \lambda \left( g(\cdot) - f(\cdot)  \right)\right]  \right\}
$$
$$
\miu\ \|\Psis\|_1\, I^{\varphi}\left[ \phis(2)^{-s} 3 \lambda\,  \left|f(\cdot)-g(\cdot)\right|\right]+ I^{\varphi}\left[ 3 \lambda \left( \K_n(g, \cdot) - g(\cdot)  \right)\right]  
$$
$$
\hskip-7.4cm +\ I^{\varphi}\left[3 \lambda \left( g(\cdot) - f(\cdot)  \right)\right]
$$
$$
\hskip-0.2cm \miu\ \left( \|\Psis\|_1 + 1\right)I^{\varphi}\left[\bar \lambda \left( g(\cdot) - f(\cdot)  \right)\right] + I^{\varphi}\left[\bar \lambda \left( \K_n(g, \cdot) - g(\cdot)  \right)\right], 
$$
for every $n \in \N^+$ sufficiently large. Now, using (\ref{densitaa}):
$$
I^{\varphi}\left[\lambda \left( \K_n(f, \cdot) - f(\cdot)  \right)\right]\, \miu\ \left( \|\Psis\|_1 + 1\right)\, \ep\ + I^{\varphi}\left[\bar \lambda \left( \K_n(g, \cdot) - g(\cdot)  \right)\right],
$$
and since Theorem \ref{lux-conv} holds, we get:
$$
I^{\varphi}\left[\bar \lambda \left( \K_n(g, \cdot) - g(\cdot)  \right)\right]\ <\ \ep,
$$
for every $n \in \N^+$ sufficiently large. In conclusion, we finally obtain:
$$
I^{\varphi}\left[\lambda \left( \K_n(f, \cdot) - f(\cdot)  \right)\right]\, \miu\ \left( \|\Psis\|_1 + 2\right)\, \ep,
$$
for every $n \in \N^+$ sufficiently large. Thus the proof follows by the arbitrariness of $\ep>0$.
\end{proof}

We can finally observe that the results proved in this section can be extended to not-necessarily non-negative functions. Indeed, for any fixed $f:\rr \to \R$, if $\inf_{\xx \in \rr} f(\xx) =: c <0$, the sequences $(\K_n(f-c, \cdot)+c)_{n \in \N+}$ turns out to be convergent to $f$ (in all the above considered convergences, provided that $f$ belongs to the corresponding spaces). For more details, see e.g., \cite{COGA1,COGA2,COVI3}.


\section{Concrete examples: Orlicz spaces and sigmoidal functions} \label{sec4}

 Here, we present examples of sigmoidal functions that can be used as activation functions in the neural network approximation process studied in the present paper. 
   
    First of all, we can consider the logistic (see Fig. \ref{fig1}, left) and hyperbolic tangent sigmoidal functions (see e.g. \cite{CACH1,CACH2,CHGE,ILKYMA1}), defined respectively by:
$$
\sigma_{\ell}(x) := (1+e^{-x})^{-1}, \hskip0.8cm \sigma_{h}(x) := 2^{-1}(\tanh(x)+1), \hskip0.8cm x \in \R.
$$
\begin{figure}[!htb]
\centering
\includegraphics[scale=1.4]{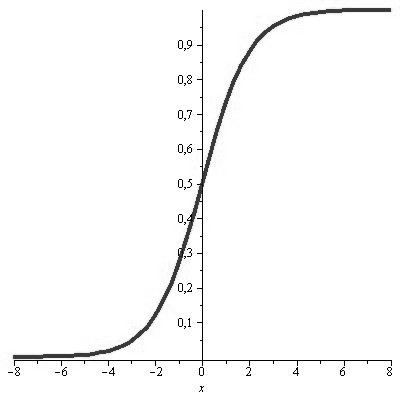}
\hskip1cm
\includegraphics[scale=1.4]{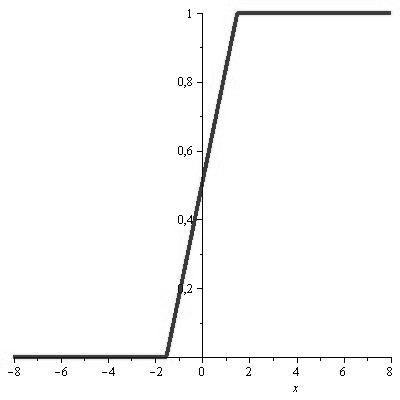}
\caption{{\small The plots of $\sigma_\ell$ (left) and $\sigma_R$ (right), i.e., the logistic and the ramp sigmoidal functions.}} \label{fig1}
\end{figure}

  Such kind of sigmoidal functions are the most used in the applications by means of neural networks, since their smoothness make them very suitable for the implementation of learning algorithms, such as the well-known back-propagation algorithm, see e.g. \cite{SXJ1,BAGR1,GN,GNSA,GOLO1,LLZ1,LWZ1,OL1}.

  Assumptions $(\Sigma 1)$, $(\Sigma 2)$, and $(\Sigma 3)$ required in order to apply the present theory, are satisfied in both the above cases. In particular, we can observe that both $\sigma_{\ell}$ and $\sigma_{h}$ have an exponential decay to zero as $x \to -\infty$, then condition $(\Sigma 3)$ is satisfied for every $\alpha >0$, see e.g., \cite{COVI2}.
  
  The above exponential decay of $\sigma_{\ell}$ and $\sigma_{h}$, it is then inherited from the corresponding univariate and multivariate density functions $\phi_{\sigma_\ell}(x)$ (see Fig. \ref{fig2}, left), $\phi_{\sigma_h}(x)$, and $\Psi_{\sigma_\ell}(\xx)$ (see Fig. \ref{fig3}, left), $\Psi_{\sigma_h}(\xx)$, for $|x| \to +\infty$, and $\|\xx\|_2 \to +\infty$, respectively.
\begin{figure}[!htb]
\centering
\includegraphics[scale=1.3]{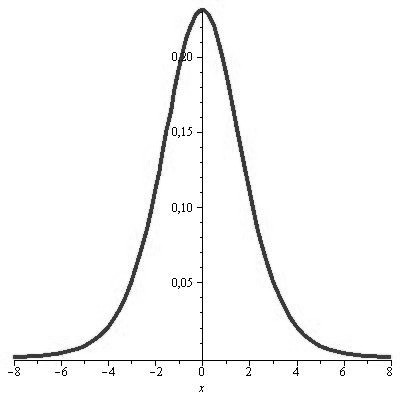}
\hskip1cm
\includegraphics[scale=1.3]{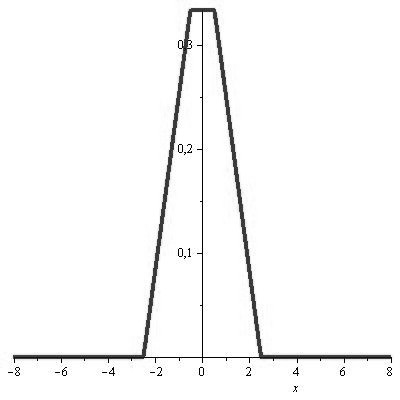}
\caption{{\small The plots of the $\phi_{\sigma_\ell}$ (left) and $\phi_{\sigma_R}$ (right).}} \label{fig2}
\end{figure}
\begin{figure}[!htb]
\centering
\includegraphics[scale=0.87]{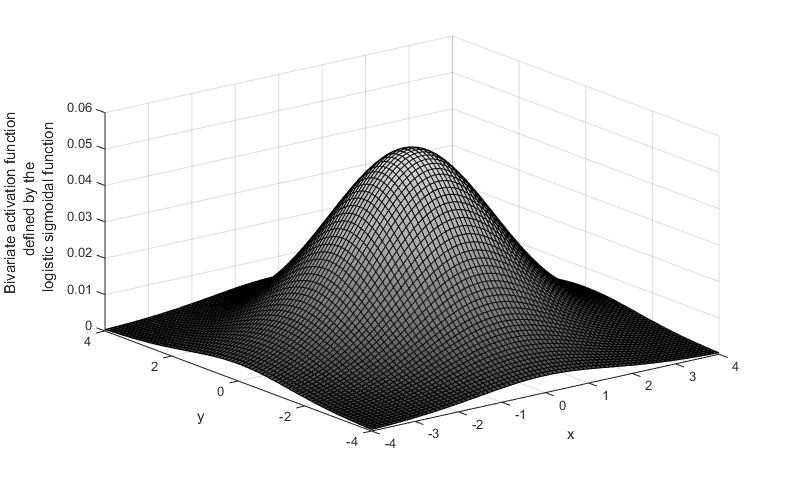}
\includegraphics[scale=0.87]{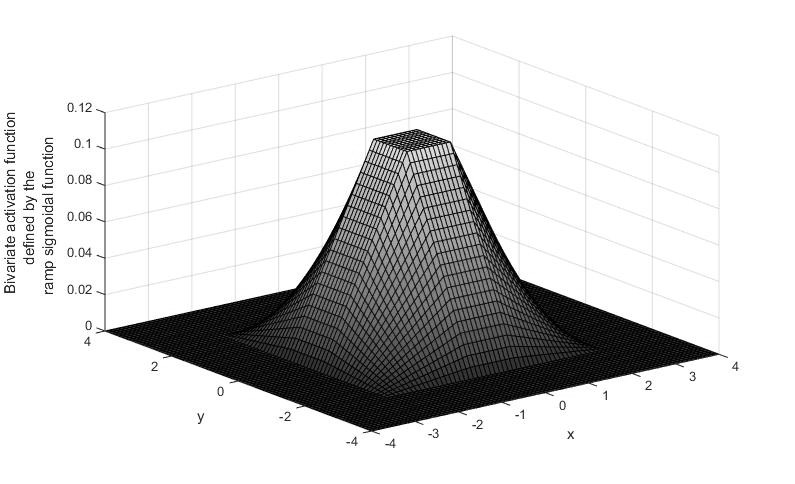}
\caption{{\small The plots of the bivariate functions $\Psi_{\sigma_\ell}$ (left) and $\Psi_{\sigma_R}$ (right).}} \label{fig3}
\end{figure}

  Further, since the above theory still holds in case of not necessarily smooth sigmoidal functions, we also provide examples of such functions. 
     
   An example of non-smooth, but continuous sigmoidal activation function is given by the well-known ramp function $\sigma_R$ (see e.g. \cite{CHGE,COSP3} and Fig. \ref{fig1}, right) defined by:
$$
\sigma_R(x)\ :=\
\left\{
\begin{array}{l}
0, \hskip2.8cm x<-3/2,\\
x/3+(1/2), \hskip0.9cm -3/2 \miu x \miu 3/2,\\ 
1, \hskip2.8cm x > 3/2.
\end{array}
\right.
$$

  In addition, it is easy to see that $(\Sigma 3)$ is satisfied for every $\alpha>0$, and the corresponding $\phi_{\sigma_R}$, $\Psi_{\sigma_R}$ (see Fig. \ref{fig2} right and Fig. \ref{fig3} right, respectively) have compact support.

   The above sigmoidal functions can be used in order to obtain the modular convergence for the operators $\K_n$ in some well-known examples of Orlicz spaces.
   
   For instance, we can consider the Orlicz spaces generated by the convex $\varphi$-function $\varphi(u) :=u^p$, $u \mau 0$, $1 \miu p < +\infty$, i.e., the well-known $L^p({\cal R})$ spaces (see e.g., \cite{BCS2}). If we consider the non-negative functions belonging to $L^p({\cal R})$ we define as above the space $L^p_+({\cal R})$. Further, we can also observe that, the above $\varphi$-function satisfies the $\Delta_2$-condition, then the modular convergence coincides with the norm-convergence.

  Here, it turns out that $I^{\varphi}[f]=\|f\|_p^p$, and Theorem \ref{th4} becomes:
\begin{corollary} \label{cor1}
For any $f \in L^p_+({\cal R})$, $1 \miu p < +\infty$, we have:
$$
\lim_{n \to+\infty}\|\K_n(f,\cdot)-f(\cdot)\|_p\ =\ 0,
$$
where $\K_n$ can be considered generated by all the above mentioned examples of sigmoidal functions.
\end{corollary}

   Other useful examples of Orlicz spaces are the interpolation spaces $L^{\alpha} \log^{\beta} L({\cal R})$ (also known as the Zygmund spaces). Such spaces, can be generated by the $\varphi$-functions $\varphi_{\alpha, \beta}(u):= u^{\alpha}\log^{\beta}(u+e)$, for $\alpha \mau 1$, $\beta>0$, $u \mau 0$. Note that, also $\varphi_{\alpha, \beta}$ satisfies the $\Delta_2$-condition, then modular convergence and norm-convergence are equivalent.  The convex modular functional corresponding to $\varphi_{\alpha, \beta}$ are: 
$$
I^{\varphi_{\alpha, \beta}}[f] := \int_{\cal R} |f(\xx)|^{\alpha} \log^{\beta}(e+|f(\xx)|)\ d\xx,\ \hskip0.5cm (f \in M({\cal R})).
$$

Obviously, if we consider the non-negative (a.e.) functions belonging to $L^{\alpha} \log^{\beta} L({\cal R})$, we can define the space $L^{\alpha}_+ \log^{\beta} L({\cal R})$.

 Finally, as last examples, we recall the exponential spaces (and the corresponding space of the a.e. non-negative functions) generated by $\varphi_{\gamma}(u)=e^{u^{\gamma}}-1$, for $\gamma>0$, $u \mau 0$. In this latter case, the $\Delta_2$-condition is not satisfied, then the norm convergence it turns out strictly stronger than the modular one.  The modular functional corresponding to $\varphi_{\gamma}$ is:
$$
I^{\varphi_{\gamma}}[f] := \int_{\cal R} (e^{|f(\xx)|^{\gamma}}-1)\ d\xx,\ \hskip0.5cm (f \in M({\cal R})).
$$
In both the last two examples, the modular inequality of Theorem \ref{th3} and the modular convergence of Theorem \ref{th4} hold, and an analogous of Corollary \ref{cor1} can be stated. 

For further examples of Orlicz spaces and sigmoidal functions, see e.g., \cite{CS1,BCS2,CMV1,RIRU1,SS1}.


\section{Applications to real world cases} \label{sec5}

In this section, we provide some applications of the above results in some concrete cases.

   We begin considering a real world case involving one-dimensional data in order to show the modeling capabilities of the above neural network operators.

     In the table of Fig. \ref{figtab}, it has been reported the data provided by the official dataset of the website {\tt www.tuttitalia.it} concerning the trend of the Italian, Foreign, and Total population living in Perugia (Italy) from 2004 to 2017. 

\begin{figure}[!htb]
\centering
\includegraphics[scale=1.02]{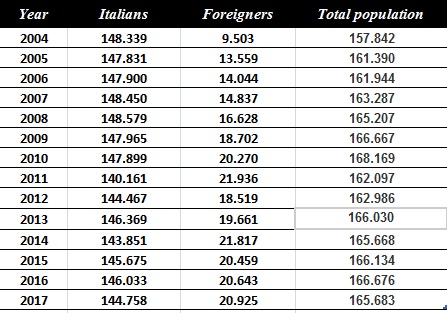}   
\caption{{\small The trend of the Italian, Foreign, and Total population living in Perugia (Italy) from 2004 to 2017.  The entries of the above table represent the average population for each considered year. The above data have been considered from the official dataset of the website {\tt www.tuttitalia.it}.}} \label{figtab}
\end{figure}         

  The entries of the above table represent the average population for each considered year.
     
     The above data have been modeled by the max-product neural network operators $\K_n(f,\cdot)$ in the one-dimensional case, and based upon the density function $\phi_{\sigma_\ell}$ generated by the logistic function $\sigma_\ell$. 
     
     More precisely, the above data have been modeled by the operators $\K_{13}(f,\cdot)$ (i.e., with $n=13$) on the interval $[a,b]=[0,1]$. Here, the years from $2004$ to $2017$ have been mapped on the nodes $k/n$, with $k=0, 1, \dots, n$, $n=13$; the averages 
$$n\, \int_{k/n}^{(k+1)/n}f(u)\, du,
$$
have been replaced with the data of the first, the second and the third column of the table in Fig. \ref{figtab}, respectively. The plots of the described neural network models have been given in Fig. \ref{fig5}. 
\begin{figure}[!htb]
\centering
\includegraphics[scale=0.27]{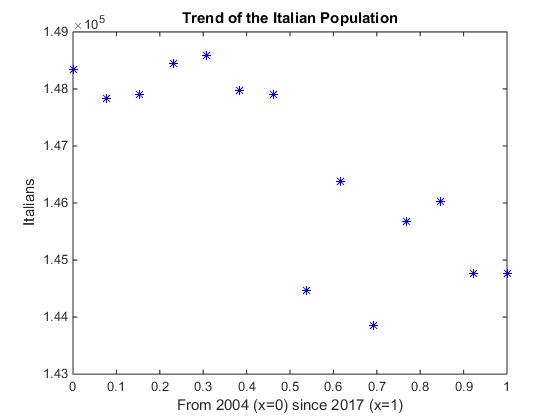}
\includegraphics[scale=0.27]{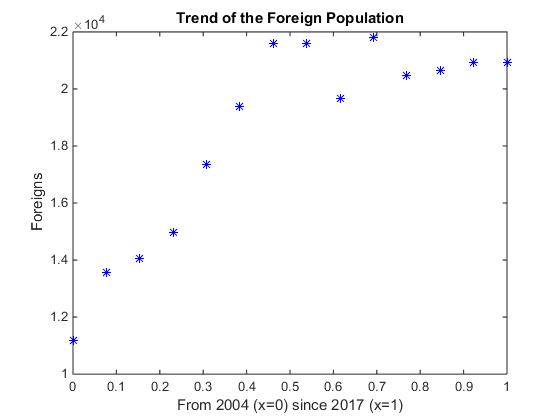}
\includegraphics[scale=0.27]{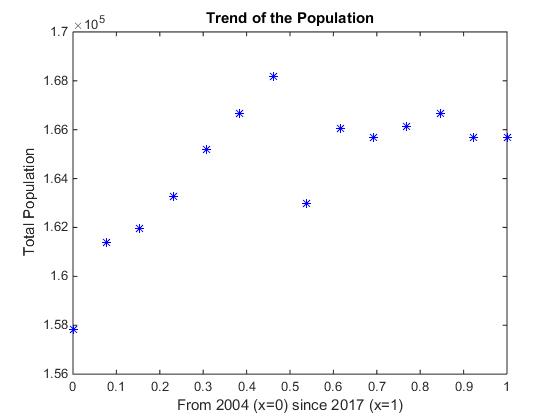}
\caption{{\small The modeled trends of the Italian, Foreign, and Total population living in Perugia (Italy) from 2004 to 2017 modeled by the max-product neural network operators $\K_{13}(f,\cdot)$.}} \label{fig5}
\end{figure}
\vskip0.5cm
     
     We stress that, the above example has been given with the main purpose to show that the max-product neural network operators can be used also to model data arising from real world problems and not only for the theoretical approximation of mathematical functions defined by a certain analytical expression.
\vskip0.5cm

  Applications in case of phenomenon involving multivariate data, can be given considering the case of image reconstruction and enhancement. This subject is very related to Kantorovich-type operators (see, e.g., \cite{ING1,ING2}) since they revealed to be suitable in case of reconstruction of not necessarily continuous data (see, e.g., Corollary \ref{cor1}) and in general, images are the most common examples of multivariate discontinuous functions. Indeed, any static gray scale image presents at the edges of the figures, jumps of gray levels in the gray scale, and from the mathematical point of view, these can be modeled by means of discontinuities. 

     For instance, using, e.g., the bi-variate version of the above max-product neural network operators we can reconstruct images on a certain domain starting from a discrete set of mean values only. 
     
     More precisely, using a model of the data similar to those used in the one-dimensional case, and knowing that by the above operators we can approximate data on a continuous domain, we can reconstruct any given image on the same nodes of the original one, or at a grid of nodes of higher dimension, in fact obtaining a reconstructed image with higher resolution with respect to the original one. \\
     
     For instance, in Fig. \ref{fig6} we have a vascular CT image (computer tomography) representing an aneurysm of the abdominal aorta artery.\\

\begin{figure}[!htb]
\centering
\includegraphics[scale=2.8]{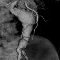}
\caption{{\small A portion of the abdominal aorta artery depicting an aneurysm of resolution $60 \times 60$ pixels (source: Santa Maria della Misericordia Hospital of Perugia
- Italy).}} \label{fig6}
\end{figure}

     The image in Fig. \ref{fig6} has a low resolution ($60 \times 60$ pixels) and the edges of the vessel are not clearly detectable, then from the diagnostic point of view, the image have a low utility.     

  First of all, a mere reconstruction of the above vascular image can be given, e.g., by the bi-variate max-product neural network operator $\K_n$ based upon the bi-variate density function $\Psi_{\sigma_\ell}$ generated by the logistic function $\sigma_\ell$, with various $n$ (see Fig. \ref{fig7}).
\begin{figure}[!htb]
\centering
\includegraphics[scale=2.2]{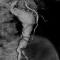}
\hskip0.8cm
\includegraphics[scale=2.2]{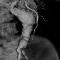}
\caption{{\small The reconstruction of the original image in Fig. \ref{fig6} by the max-product operators $\K_n$ with $n=5$ (on the left) and $n=10$ (on the right).}} \label{fig7}
\end{figure}
Finally, in Fig. \ref{fig8} are shown the reconstruction of the original image of Fig. \ref{fig6} by the max-product operators $\K_n$ with $n=5$ and $n=20$, respectively obtained by a double resolution, i.e., $120 \times 120$ pixels, with respect to the original one. In this sense, the images in Fig. \ref{fig8} have been  enhanced with respect to that of Fig. \ref{fig6}.
\begin{figure}[!htb]
\centering
\includegraphics[scale=1.15]{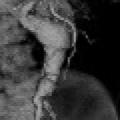}
\hskip0.3cm
\includegraphics[scale=1.15]{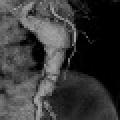}
\caption{{\small The reconstruction of the original image in Fig. \ref{fig6} by the max-product operators $\K_n$ with $n=5$ (on the left) and $n=20$ (on the right). Note that, the above reconstructions have been given by a double resolution (i.e., $120 \times 120$ pixels) with respect to the original image.}} \label{fig8}
\end{figure}
    Indeed, we can observe that here the contours and the edges of the vessels are more defined with respect to the original one. 
     Hence, NN operators open also an application field in image reconstruction and enhancement.\\
      
\noindent
{\bf Conflict of interest}
The authors declare that they have no conflict of interest.

\noindent {\bf Ethical statement}
Ethical approval was waived considering that the CT images analyzed were anonymized and the
results did not influence any clinical judgment.

\section*{Acknowledgments}

\begin{itemize}
\item 
The authors are members of the Gruppo Nazionale per l’Analisi Matematica, la Probabilitàe le loro Applicazioni (GNAMPA) of the Istituto Nazionale di Alta Matematica (INdAM).Moreover, the first and the second authors of the paper have been partially supported within the 2018 GNAMPA-INdAM Project ‘‘Dinamiche non autonome, analisi reale eapplicazioni,’’ while the second and the third authors within the project: Ricerca di Base 2017 dell’Università degli Studi di Perugia ‘‘Metodi di teoria degli operatori e di Analisi Reale per problemi di approssimazione ed applicazioni.

\item
This is a post-peer-review, pre-copyedit version of an article published in  Neural Computing and Applications. The final authenticated version is available online at: http://dx.doi.org/10.1007/s00521-018-03998-6.
\item The authors would like to thank the referees for their useful suggestions which led us to insert the section devoted to real world applications.
\end{itemize}

\small
%


\end{document}